\documentclass[12pt]{amsart}

\usepackage{amsmath,amsthm,amssymb,amsfonts,verbatim}
\usepackage[colorlinks,
            linkcolor=blue,
            anchorcolor=blue,
            citecolor=blue
            ]{hyperref}
\usepackage[hmargin=1.2in,vmargin=1.2in]{geometry}
\def\PP{{\mathbb{P}}}

\def\mG{{\mathcal G}}

\def \a {\alpha}
\def \b {\beta}

\def \s {\sigma}

\def \F {{\mathbb F}}

\def \deg {{\rm deg}}

\def \Aut {{\rm Aut}}
\def \Gal {{\rm Gal}}

\def \G {{\mathcal G}}

\newtheorem{theorem}{Theorem}[section]

\newtheorem{proposition}[theorem]{Proposition}
\newtheorem{lemma}[theorem]{Lemma}
\newtheorem{definition}[theorem]{Definition}
\newtheorem{example}{Example}

\newtheorem{open}{Open Problem}

\begin{document}

\title[automorphism group]{The asymptotic behavior of  automorphism groups of function fields over finite fields}
\thanks{The first author is partially supported by the National Natural Science Foundation of China under Grant 11501493, the Shuangchuang Doctor Project of Jiangsu Province and the Young Visiting Scholar Project of China Scholarship Council.}
\author{Liming Ma}\address{School of Mathematical Sciences, Yangzhou University, Yangzhou China
225002}\email{lmma@yzu.edu.cn}
\author{Chaoping Xing} \address{Division of Mathematical Sciences, School of Physical \& Mathematical Sciences,
Nanyang Technological University, Singapore
637371}\email{xingcp@ntu.edu.sg}
\maketitle

\begin{abstract} The purpose of this paper is to investigate the asymptotic behavior of automorphism groups of function fields when genus tends to infinity. 

Motivated by applications in coding and cryptography, we consider   the maximum size of abelian subgroups of the  automorphism group $\Aut(F/\F_q)$ in terms of genus ${g_F}$ for  a function field $F$ over a finite field $\F_q$. Although the whole group $\Aut(F/\F_q)$ could have size $\Omega({g_F}^4)$, the maximum size $m_F$ of abelian subgroups of the  automorphism group $\Aut(F/\F_q)$ is upper bounded by $4g_F+4$ for $g_F\ge 2$. In the present paper, we study the asymptotic behavior of $m_F$ by defining $M_q=\limsup_{{g_F}\rightarrow\infty}\frac{m_F \cdot \log_q m_F}{{g_F}}$, where $F$ runs through all function fields over $\F_q$. We show that $M_q$ lies between $2$ and $3$ (or $4$) for odd characteristic (or for even characteristic, respectively). This means that $m_F$ grows much more slowly than  genus does asymptotically.

The second part of this paper is to study the maximum size $b_F$ of subgroups of $\Aut(F/\F_q)$ whose order is coprime to $q$. The Hurwitz bound gives an upper bound $b_F\le 84(g_F-1)$ for every function field $F/\F_q$ of genus $g_F\ge 2$. We investigate the asymptotic behavior of $b_F$  by defining ${B_q}=\limsup_{{g_F}\rightarrow\infty}\frac{b_F}{{g_F}}$, where $F$ runs through all function fields over $\F_q$. Although the Hurwitz bound shows ${B_q}\le 84$,  there are no lower bounds on $B_q$ in literature. One does not even know if ${B_q}=0$. For the first time, we show that ${B_q}\ge  2/3$ by explicitly constructing some towers of function fields in this paper.
\end{abstract}

\section{Introduction}
Let $q$ be a prime power and let $\F_q$ be the finite field with $q$ elements.
For a global function field $F/\mathbb{F}_q$, we denote by $\Aut(F/\mathbb{F}_q)$ the automorphism group of $F$ over $\mathbb{F}_q$, that is, $$ \Aut(F/\mathbb{F}_q)=\{\sigma: F\rightarrow F| \sigma \text{ is an } \mathbb{F}_q\text{-automorphism of } F\}.$$

Due to demand from applications such as coding theory and cryptography \cite{CPX15,GX15,GXC17,KS17,St90, We98, Xi95}, there has been a lot of research on automorphism groups of univariate function fields over finite fields (i.e., global function fields) (see \cite{GK17, He78,HKT08, KM16, MXY16, St06}). One of the problems in this topic  is to look at the relation between the size of the  automorphism group $\Aut(F/\F_q)$ and the genus ${g_F}$ for a function field $F$ over a finite field $\F_q$.

For a function field $F$ over $\F_q$, the full constant field of $F$ is defined to be the subfield of $F$ whose elements are algebraic over $\F_q$. In this paper, we always mean that $\F_q$ is the full constant field whenever we write $F/\F_q$.
For a function field $F/\F_q$ of genus $g_F\geq 2$, if $|\Aut(F/\mathbb{F}_q)|\geq 8g_F^3$, then $F$ is isomorphic to one of the four exceptions: two hyperelliptic function fields, the Hermitian function field and the Suzuki function field \cite{He78}. In particular, for the Hermitian function field $H=\F_q(x,y)$ defined by $y^r+y=x^{r+1}$ with $q=r^2$, the automorphism group has size $r^3(r^2-1)(r^3+1)$. Thus, in the Hermitian function field case, one has $|\Aut(H/\mathbb{F}_q)|\ge 16g_H^4$.

A natural question is how $|\Aut(F/\F_q)|$ behaves when $g_F$ tends to infinity for a fixed $q$?
For a fixed finite field $\F_q$, the genus $g_F$ for any function field $F/\F_q$ in the above four classes of the exceptional function fields  with $|\Aut(F/\mathbb{F}_q)|\geq 8g_F^3$ are upper bounded. This implies that, for a fixed $q$, one has
$$\limsup _{{g_F}\rightarrow \infty} \frac{|\Aut(F/\mathbb{F}_q)|}{{g_F}^3}\leq 8,$$
where $F$ runs through all function fields over $\F_q$. This leads to the following open problem.
\begin{open} 
How does $|\Aut(F/\F_q)|$ behave when $g_F$ tends to infinity for a fixed $q$? How does one properly define an asymptotic quantity of $|\Aut(F/\F_q)|$ in terms of genus $g_F$?
\end{open}
The second question refers to growing speed of $|\Aut(F/\F_q)|$ in terms of genus $g_F$. For instance, we do not know if $|\Aut(F/\F_q)|$ grows linearly or quadratically in genus $g_F$.

On the other hand, for some  applications in coding theory and cryptography \cite{CPX15,GX15,GXC17}, one requires a large abelian (or even cyclic) subgroup of $\Aut(F/\F_q)$. This motives us to study  the relation between  the maximum size of abelian subgroups of the  automorphism group $\Aut(F/\F_q)$ and the genus ${g_F}$ for a function field $F$ over a finite field $\F_q$.
For a function field $F/\F_q$, we define the quantity
\[m_F=\max\{|\mathcal{G}|:\; \mbox{$\mathcal{G}$ is an abelian subgroup of } \Aut(F/\mathbb{F}_q)\}.\]
Although $|\Aut(F/\F_q)|$ can be as large as $16{g_F}^4$, the quantity $m_F$ is much smaller. In fact, it was proved in \cite[Theorem 11.79]{HKT08} that $m_F$ is at most linear in ${g_F}$. More precisely we have that for any function field $F$ with genus ${g_F}\ge 2$, one has
$$m_F \leq
\begin{cases}
4{g_F}+4, & \mbox{ for char} (\mathbb{F}_q)\neq 2,\\
4{g_F}+2, & \mbox{ for char} (\mathbb{F}_q)= 2.
\end{cases}$$
For the Hermitian function field defined above, we have $m_F\ge q-1\ge 2{g_F}$ \cite{GSX00}. This implies that $m_F$ can indeed be linear in ${g_F}$. But the question is how $m_F$ behaves asymptotically as ${g_F}$ tends to infinity. Can $m_F$ grow linearly with ${g_F}$ when ${g_F}\rightarrow\infty$? Our result shows that $m_F$ grows much more slowly than $g_F$. In fact, in this paper we show that ${g_F}$ grows at least as fast as $\Omega(m_F\log m_F)$ asymptotically.
Thus, to study the asymptotic behavior of $m_F$ as ${g_F}$ tends to infinity,  we define the asymptotic quantity
 $$M_q=\limsup_{{g_F}\rightarrow \infty} \frac{m_F\cdot \log_q m_F}{{g_F}}$$
for every prime power $q$.     We will see later that $M_q$ is a positive constant.  The main purpose of the first part of this paper is to find some reasonable lower and upper bounds on $M_q$.

The second part of this paper is to study the maximum size $b_F$ of subgroups of $\Aut(F/\F_q)$ whose order is coprime to $q$, i.e.,
\[b_F=\max\{|\mathcal{G}|:\; \mathcal{G}\le \Aut(F/\mathbb{F}_q) \mbox{ and} \; \gcd(|\mathcal{G}|, q)=1 \}.\]
The Hurwitz bound gives an upper bound $b_F\le 84(g_F-1)$ for every function field $F/\F_q$ of genus $g_F\ge 2$ \cite{HKT08, Hu93}. The  Hermitian function field $H$ gives $b_H\ge q-1$ \cite{GSX00}. However, in this case both $b_H$ and $g_H$ depend on $q$. Therefore, the  Hermitian function field does not provide any information on the asymptotic behavior of $b_F$. In this paper, we show that, over a fixed $q$, $b_F$ can grow linearly in $g_F$. To prove our result, we introduce the following asymptotic quantity
$${B_q}=\limsup_{g_F \rightarrow \infty}\frac{b_F}{g_F},$$
 where $F$ runs through all function fields over $\F_q$.
 It follows  from the Hurwitz bound that ${B_q}\le 84$. 
 As far as we know, there are no lower bounds on ${B_q}$ in literature. One does not even know if ${B_q}=0$. In this paper, we show that ${B_q}\ge 2/3$ by explicitly constructing some towers of function fields. This means that $b_F$ can grow linearly with ${g_F}$ when ${g_F}\rightarrow\infty$.

The paper is organized as follows. In Section 2, we will introduce some preliminaries on function fields including  Hilbert's ramification theory, conductor, cyclotomic function fields, ray class fields and the Chebotarev Density Theorem. Section 3 is devoted to prove lower and upper bounds on  $M_q$. In the last section, we prove a lower bound on ${B_q}$ by explicitly constructing two towers of function fields.

\section{Premilinaries}
\label{premilinary}
Let $F/\mathbb{F}_q$ be a global function field of genus $g_F\geq 2$.
For a subgroup $\mathcal{G}$  of the automorphism group of $F/\mathbb{F}_q$,
denote by ${F^{\mathcal G}}$ the fixed subfield of $F$ with respect to $\mathcal{G}$, that is,
$$F^{\mathcal G}=\{z\in F| \sigma(z)=z \text{ for any } \sigma \in \mathcal{G}\}.$$
Then $F/F^{\mathcal G}$ is a Galois extension with the Galois group $\Gal(F/F^{\mathcal G})={\mathcal G}$.

\subsection{Hilbert's ramification theory}
The Hurwitz genus formula yields
$$2{g_F}-2=|{\mathcal G}|\cdot (2g(F^{\mathcal G})-2)+\deg \text{ Diff}(F/F^{\mathcal G}),$$
where $\text{Diff}(F/F^{\mathcal G})$ stands for the different of the extension $F/F^{\mathcal G}$  (see \cite[Theorem 3.4.13]{St09}). 

Let $\PP_F$ denote the set of places of $F$.
For a place $P\in \PP_F$ and a place $Q$ with $Q=P\cap F^{\mathcal G}$ being the restriction of $P$ to $F^{\mathcal G}$,
we denote by $d_P(F/F^{\mG}), \ e_P(F/F^{\mG})$ (or $d(P|Q)$ and $e(P|Q)$, respectively)  the different exponent and ramification index of $P|Q$, respectively.
Then the different of $F/F^{\mathcal G}$ is given by
$$\text{Diff}(F/F^{\mathcal G})=\sum_{P\in\PP_F} d_P(F/F^{\mG}) P.$$
If $P|Q$ is unramified or tamely ramified, then $d_P(F/F^{\mG})=e_P(F/F^{\mG})-1$ by Dedekind's Different Theorem \cite[Theorem 3.5.1]{St09}. However, if
 $P|Q$ is wildly ramified, that is, $e_P(F/F^{\mG})$ is divisible by $\text{char}(\mathbb{F}_{q})$, then it is more complicated to calculate the different exponent $d_P(F/F^{\mG})$. One way to find the different exponent $d_P(F/F^{\mG})$ is through ramification groups and Hilbert's Different Theorem.

The $i$-th ramification group ${\mathcal G}_i(P)$ of $P|Q$ for each $i\ge -1$ is defined by
$${\mathcal G}_i(P)=\{ \sigma\in {\mathcal G}| v_P(\sigma(z)-z) \ge i+1 \text{ for all } z\in \mathcal{O}_P\},$$
where $\mathcal{O}_P$ stands for the integral ring of $P$ in $F$ and $v_P$ is the normalized discrete valuation of $F$ corresponding to the place $P$.
If $P|Q$ is wildly ramified, then the different exponent $d_P(F/F^{\mG})$ is
$$d_P(F/F^{\mG})=\sum_{i=0}^{\infty} \Big{(} |{\mathcal G}_i(P)|-1\Big{)}$$
by  Hilbert's Different Theorem \cite[Theorem 3.8.7]{St09}. Let $a_P(F/F^\mG)$ be the least non-negative integer $l$ such that the  ramification groups $\mathcal{G}_i(P)$ are trivial for all $i\ge l$. Then we have $$d_P(F/F^{\mG})=\sum_{i=0}^{a_P(F/F^\mG)-1} \Big{(}|\mathcal{G}_i(P)|-1\Big{)}\geq a_P(F/F^\mG).$$

\subsection{Conductor}
\label{conductor}
For a real number $x$, we extend the above definition of ramification groups to  real numbers $x\ge -1$ by putting $\mathcal G_x(P)=\mathcal G_{\lceil x \rceil}(P)$, where $\lceil x \rceil$ is the least integer $\ge x$. Let $g_x$ be the order of the $x$-th ramification group $\mathcal G_x(P)$.
Define the function $\varphi(x)$ for $x\ge -1$ by putting
$$\varphi(x)=\int_0^{x} \frac{1}{[\mathcal G_0(P): \mathcal G_t(P)]} dt,$$
where $[\mathcal G_0(P): \mathcal G_t(P)]=[\mathcal G_t(P):\mathcal G_0(P)] ^{-1}$ if $t\le 0$. Explicitly, we have
$$\varphi(x)=\frac{1}{g_0}\Big{(}g_1+g_2+\cdots g_{\lfloor x\rfloor}+(x-\lfloor x \rfloor) g_{\lfloor x\rfloor+1}\Big{)}$$
for $x> 0$ and $\varphi(x)=x$  for $-1\le x\le 0$.
Then the function $\varphi$ is continuous, piecewise linear, strictly monotone increasing, and concave on $[-1, \infty)$ (see \cite[Chapter XI.2]{AT67}).

Now we define the function $\psi$ to be the inverse function of $\varphi$ and we define the $v$-th upper index ramification group by
$$ \mathcal G^v(P):=\mathcal G_{\psi(v)}(P).$$
Then we have $ \mathcal G^{-1}(P)= \mathcal G_{-1}(P)$, $ \mathcal G^{0}(P)= \mathcal G_{0}(P)$ and $ \mathcal G^{v}(P)=\{\mbox {id}\}$ for sufficiently large $v$.

The conductor exponent $c_P(F/{F^{\mG}})$ can be equivalently defined to be the least non-negative integer $k$ such that the upper index ramification groups $\mathcal{G}^v(P)$   is trivial for all $v\ge k$ (see \cite[Definition 3.3.3 and Theorem 3.8.10]{Ne86}).
Moreover, we have the following relations
\begin{equation} \label{relationcde}
c_P(F/{F^{\mG}})=\frac{d_P(F/{F^{\mG}})+a_P(F/{F^{\mG}})}{e_P(F/{F^{\mG}})}\le 2\cdot \frac{ d_P(F/{F^{\mG}})}{e_P(F/{F^{\mG}})}.
\end{equation}
It is well known that $c_P(F/{F^{\mG}})=0$ if and only if $P$ is unramified in $F/{F^{\mG}}$, $c_P(F/{F^{\mG}})=1$ if and only if $P$ is tamely ramified in $F/{F^{\mG}}$, and $c_P(F/{F^{\mG}})\geq 2$ if and only if $P$ is wildly ramified in $F/{F^{\mG}}$ (see \cite[Theorem 2.3.4]{NX01}).

Define the conductor of $F/{F^{\mG}}$ by
$$\text{Cond}(F/{F^{\mG}})=\sum_{P\in \PP_F} c_P(F/{F^{\mG}})P.$$
It is clear that the conductor of $F/{F^{\mG}}$ measures ramification of the extension $F/{F^{\mG}}$.

\subsection{Cyclotomic function fields} 

In this subsection, we briefly review some of the fundamental notions and results of cyclotomic function fields. The theory of cyclotomic function fields was developed in the language of function fields by Hayes (see \cite{Ha74, NX01}).

Let $q$ be a prime power. Let $x$ be an indeterminate over $\F_q$, $R=\F_q[x]$ the polynomial ring, $k=\F_q(x)$ the quotient field of $R$, and $k^{ac}$ the algebraic closure of $k$. Let $\varphi$ be the endomorphism given by $$\varphi(z)=z^q+xz $$  for all $z\in k^{ac}$. Define a ring homomorphism
$$R\rightarrow \text{End}_{\mathbb{F}_q}(k^{ac}), f(x)\mapsto f(\varphi).$$
Then the $\F_q$-vector space of $k^{ac}$ is made into an $R$-module by introducing the following action of $R$ on $k^{ac}$, namely,
$$ z^{f(x)}=f(\varphi)(z)$$  for all $f(x)\in R$ and $z\in k^{ac}$. For a nonzero polynomial $M\in R$, we consider the set of $M$-torsion points of $k^{ac}$ defined by $$\Lambda_M=\{z\in F^{ac}| z^M=0\}.$$
In fact, $z^M$ is a separable polynomial of degree $q^d,$ where $d=\deg(M)$. The cyclotomic function field over $k$ with modulus $M$ is defined by the subfield of $k^{ac}$ generated over $k$ by all elements of $\Lambda_M$, and it is denoted by $k(\Lambda_M)$. In particular, we list the following facts:

\begin{proposition}
\label{genusofcyclotomic}
Let $P$ be a monic irreducible polynomial of degree $d$ in $R$ and let $n$ be a positive integer. Then
\begin{itemize}
\item[\rm (i)] $[k(\Lambda_{P^n}):k]=\phi(P^n)$, where $\phi(P^n)$ is the Euler function of $P^n$, i.e., $\phi(P^n)=q^{(n-1)d}(q^d-1)$.
\item[\rm (ii)] ${\rm Gal}(k(\Lambda_{P^n})/k) \cong (\mathbb{F}_q[x]/(P^{n}))^*.$ The Galois automorphism $\sigma_f$ associated to $\overline{f}\in (\mathbb{F}_q[x]/(P^{n}))^*$ is determined by $\sigma_f(\lambda)=\lambda^f$ for $\lambda\in \Lambda_{P^{n}}$.
\item[\rm (iii)]The zero place of $P$ in $k,$ also denoted by $P$, is totally ramified in  $k(\Lambda_{P^n})$ with different exponent $d_P(k(\Lambda_{P^n})/k)=n(q^d-1)q^{d(n-1)}-q^{d(n-1)}$. All other finite places of $k$ are unramified in $k(\Lambda_{P^n})/k$.
\item[\rm (iv)]The infinite place $\infty$ of $k$ splits into $\phi(P^n)/(q-1)$ places of $k(\Lambda_{P^n})$  and the ramification index $e_{\infty}(k(\Lambda_{P^n})/k)$ is equal to $q-1$. In particular, $\mathbb{F}_q$ is the full constant field of  $k(\Lambda_{P^n})$.
\item[\rm (v)] The genus of $k(\Lambda_{P^n})$ is given by $$2g(k(\Lambda_{P^n}))-2=q^{d(n-1)}\Big{[}(qdn-dn-q)\frac{q^d-1}{q-1}-d\Big{]}.$$
\end{itemize}
\end{proposition}

\subsection{Ray class fields} Let $E/\F_q$ be a global function field.
Let $\mbox{Div}^0(E)$ be the subgroup of the divisor group $\mbox{Div}(E)$ that consists of all divisors of $E$ of degree $0$. The principal divisor group $\mbox{Princ}(E)=\{\mbox{div}(x): x\in E^*\}$ is a subgroup of  $\mbox{Div}^0(E)$. The factor group ${\rm Cl}(E):=\mbox{Div}^0(E)/\mbox{Princ}(E)$ is called the divisor class group of degree zero of $E$ and the cardinality of Cl$(E)$ is called the divisor class number of $E$, which is denoted by $h_E$.

We fix a place $\infty$ of $E/\mathbb{F}_q$ of degree $t$. Denote by $S_\infty$ the set $\mathbb{P}_E\setminus \{\infty\}$. Let $A$ be the holomorphy ring $\mathcal{O}_{S_\infty}$, i.e.,
$$A=\{ x\in E: \; v_Q(x)\geq 0 \mbox{ for all } Q\neq \infty\}.$$
Let Fr$_{\infty}$ and Princ$_{\infty}$ denote the fractional $S_\infty$-ideal group and principal $S_\infty$-ideal group of $A$, respectively. Then the fractional ideal class group Cl$(A):=\mbox{Fr}_{\infty}/\mbox{Princ}_{\infty}$ of $A$ is a finite abelian group with cardinality $h(A)=t\cdot h_E$ \cite{R73}.

Let $D=\sum_{Q}v_Q(D)Q$ be a positive divisor of $E$ with $\infty\notin \text{supp}(D)$. For $x\in E^*$,  $x\equiv 1 (\mbox{mod } D)$  means that $x$ satisfies the following condition:
$$v_Q(x-1)\geq v_Q(D) \text{ for each } Q\in \mbox{supp}(D).$$
Let Fr$_{D,\infty}$ be the subgroup of Fr$_{\infty}$ consisting of the $S_\infty$-ideals that are relatively prime to $D$, that is, $$\mbox{Fr}_{D,\infty}=\{I\in \mbox{Fr}_{\infty}:\; v_Q(I)=0 \mbox{ for all } Q\in \mbox{supp}(D)\}.$$
Define the subgroup Princ$_{D,\infty}$ of Fr$_{D,\infty}$ by
$$\mbox{Princ}_{D,\infty}=\{(xA):\; x\in E^*, x\equiv 1 (\mbox{mod } D)\}.$$
The factor group $\mbox{Fr}_{D,\infty}/\mbox{Princ}_{D,\infty}$ is called the $S_{\infty}$-ray class group modulo $D$.
It is a finite group and denoted by Cl$_D(A)$. If $D=0$, then  Cl$_D(A)=\mbox{Cl}(A).$

The $S_\infty$-ray class field modulo $D$, denoted by $E_{\infty}^D$, is constructed as a finite abelian extension of $E$ corresponding to a certain open subgroup of the id\`{e}le class group of $E$ with finite index in which the Galois group is isomorphic to Cl$_D(A)$ (see \cite[Section 2.5]{NX01}). The ray class field $E_{\infty}^D$ is the largest finite abelian extension $F$ of $E$ such that the place $\infty$ splits completely in $F/E$ and the conductor divisor  Cond$(F/E)\leq D.$
The degree and genus formula of the ray class field $E_{\infty}^D$ can be found from \cite{Au00}. For a positive divisor  $D=\sum_{j=1}^{s}c_jQ_j$, we denote by $\phi(D)$ the Euler function of $D$, i.e.,  $\phi(D)=\prod_{j=1}^s(q^{\deg (Q_j)}-1)q^{(c_j-1)\deg (Q_j)}$.

\begin{proposition} [see \cite{Au00}]
Let $\infty$ be a place of $E$ of degree $t>0$. Let $D=\sum_{j=1}^{s}c_jQ_j$ be a positive divisor of $E$ with $\infty \notin \mbox{supp}(D)$. Let $h_E$ and $g_E$ be the class number  and the genus of $E$, respectively. Then we have
\begin{itemize}
\item[(1)] $[E_{\infty}^D:E]=\frac1{q-1}\cdot h_E\cdot t \cdot \phi(D).$
\item[(2)]  The genus of $E_{\infty}^D$ is $$g(E_{\infty}^D)=1+\frac{h_E}{2q-2}\big{[}\phi(D)(2g_E-2+\deg(D))-w\big{]}$$
where $w=[\phi(D)/\phi(Q_1)-q-2]\cdot \deg(Q_1)$ if $D$ is the multiple of a single place $Q_1$, and $w=\sum_{j=1}^s \phi(D)\deg (Q_j)/ \phi(Q_j)$ otherwise.
\end{itemize}
\end{proposition}

\subsection{Chebotarev Density Theorem}
Let $F/{E}$ be a Galois extension of degree $m$ of global function fields over the same full constant field $\mathbb{F}_q$. 
For a place $P$ of $F$ lying over $Q$ of ${E}$, let $[\frac{F/{E}}{P}]$ be the Frobenius automorphism of $P$ over $Q$. Then for any automorphism $\sigma \in \Gal(F/{E})$, the Frobenius automorphism of $\sigma(P)$ is $\s [\frac{F/{E}}{P}] \s^{-1}$. The conjugacy class
$$\Big{\{}\s \big{[}\frac{F/{E}}{P}\big{]} \s^{-1}: \s \in \Gal(F/{E})\Big{\}}$$ is determined by $Q$. Hence, we denote this conjugacy class by $[\frac{F/{E}}{Q}]$. The Chebotarev density theorem, in many different forms, gives
an equidistribution result for the occurrence of conjugacy classes as the Frobenius
class of places. An example of such a theorem is \cite[Theorem 9.13B]{R02}.

\begin{theorem}{\bf (Chebotarev Density Theorem)}
Let $F/{E}$ be a finite Galois extension of global function fields and let $\mathcal{C}$ be a conjugacy class in $\Gal(F/{E})$. Let $U_E$ denote the set of places of $E$ that are unramified in $F$. Then, for every positive integer $k\ge 1$, one has
 \[\left|\left\{Q\in U_{E}:\; [\frac{F/{E}}{Q}]=\mathcal{C},\; \deg(Q)=k\right\}\right|=\frac{|\mathcal{C}|}{[F:{E}]}\times\frac{q^k}{k}+O\left(\frac{q^{k/2}}k\right).\]
\end{theorem}

The above Chebotarev Density Theorem tells us the asymptotic behavior about how  places of $E$ split in $F$. For our purpose, we need an explicit version of the Chebotarev Density Theorem. Let $x$ be a separating transcendent element of ${E}$ over $\mathbb{F}_q$. Let $d=[{E}:\F_q(x)]$. For a conjugacy class $\mathcal{C}$ of $\Gal(F/{E})$, let $N_k(F/{E}, \mathcal{C})$ denote the number of places $Q$ of degree $k$ in ${E}$ which are unramified in both $F/{E}$ and ${E}/\F_q(x)$ such that $[\frac{F/{E}}{Q}]=\mathcal{C}$. Then the following result holds true from \cite[Proposition 6.4.8]{FJ08} and \cite{MS94}, or from \cite[Theorem 3.7]{GX15}.
\begin{proposition}
\label{degdistribution} Notations are given as  above. Then one has
$$\Big{|} N_k(F/{E}, \mathcal{C})-\frac{|\mathcal{C}|}{km}q^k\Big{|}\leq
\frac{2|\mathcal{C}|}{km}(m+g_F)q^{k/2}+m(2g_{E}+1)q^{k/4}+g_F+dm. $$
\end{proposition}

\section{Asymptotic behavior of abelian subgroups of  automorphism groups}
The main result of this section is to show that $M_q$ is between $2$ and $3$ for odd characteristic (or $4$ for even characteristic).

\subsection{Finding a suitable ray class field}
Let $F/\F_q$ be a global function field of genus $g_F\ge 2$ and let $\mG$ be an abelian subgroup of $\Aut(F/\F_q)$ with $m_F=|\mG|$. The main purpose of this subsection  is to show that $F$ is contained in the ray class group $\left(F^{\mG}\right)_{\infty}^D$ for some lower-degree place $\infty$ and lower-degree positive divisor $D$.  The idea works as follows. Assume that there exists a place $\infty$ of $F^\mG$ such that $\infty$ splits completely in $F/F^\mG$.
Let $D$ be the conductor divisor of $F/{F^\mG}$, then $F$ is a subfield of $\left(F^{\mG}\right)_{\infty}^D$ from the Conductor Theorem \cite[Theorem 2.5.4]{NX01}. Hence, we need to find a place $\infty$ of $F^\mG$ that splits completely in $F/F^\mG$. This can be done via the Chebotarev Density Theorem.

For convenience, we denote by $E$ the function field $F^\mG$.

\begin{proposition}
\label{upperboundt} 
Put $t=\lceil  6 \log_q g_F+18\rceil$. Then there exists a place $\infty$ of degree $t$ of $E$ such that $\infty$ splits completely in $F/E$.
\end{proposition}
\begin{proof}     
For any place $P\in \mathbb{P}_E$ and each integer $k\ge 2g_{E}$, there exists an element $x\in {E}$ with pole divisor $(x)_{\infty} =k \cdot P$ 
from Riemann's Theorem (see \cite[Proposition 1.6.6]{St09}).
Hence, we can find a separating transcendence element $x$ of ${E}/\mathbb{F}_q$ with $d=[{E}:\mathbb{F}_q(x)]=\ell\cdot t$ for some positive integer $\ell\le 2g_{E}+1$ (note that if $2g_E$ is divisible by the characteristic, we choose $\ell=2g_E+1$, otherwise we let $\ell=2g_E$).

Let $\mathcal{C}$ be the conjugacy class containing the identity automorphism of $F/{E}$.
Then it is clear that it is now sufficient to prove
\[N_t(F/E,{\mathcal C})\ge 1.\]
By Proposition \ref{degdistribution}, we have the following inequality
$$  N_t(F/{E}, \mathcal{C}) \geq \frac{1}{m_Ft}q^t-\left( \frac{2}{m_Ft}(m_F+g_F)q^{t/2}+m_F(2g_{E}+1)q^{t/4}+g_F+dm_F\right).$$
 Thus, we need to prove that the right-hand side of the above inequality is positive.

First we note that we have (i) $m_F=[F:E]\le  4g_F+4$ (note that $\mG$ is an abelian subgroup of $\Aut(F/\F_q)$); (ii) $g_E\le g_F$; and (iii) $d=\ell t\le (2g_E+1)t\le (2g_F+1)t$. Thus, we have
\[
\frac{2}{m_Ft}(m_F+g_F)q^{t/2}\le \frac{2}{m_Ft}\left(4g_F+4+g_F\right)q^{t/2}\le\frac14\cdot \frac{1}{m_Ft}q^t;
\]
\[m_F(2g_{E}+1)q^{t/4} \le (4g_F+4)(2g_{F}+1)q^{t/4} \le \frac14\cdot \frac{1}{(4g_F+4)t}q^t \le\frac14\cdot \frac{1}{m_Ft}q^t;\]
\[g_F<\frac14\cdot \frac{1}{(4g_F+4)t}q^t\le \frac14\cdot \frac{1}{m_Ft}q^t;\]
and \[dm_F\le (2g_F+1)t(4g_F+4)\le\frac14\cdot \frac{1}{(4g_F+4)t}q^t\le \frac14\cdot \frac{1}{m_Ft}q^t.\]
This completes the proof.
\end{proof}

\subsection{Bounds on $M_q$}
Again, we assume that $F/\F_q$ is a function field of genus $g_F\ge 2$ and let $\mG$ be an abelian subgroup of $\Aut(F/\F_q)$ with $|\mG|=m_F$. Put $E=F^\mG$.
By Proposition \ref{upperboundt}, there exists a place $\infty$ of $E$ with degree $t=\lceil  6 \log_q g_F+18\rceil$.
Let the effective divisor $D=\sum_{i=1}^{s}c_iQ_i$ be the conductor of $F/E$. Then $F$ is a subfield of $E_{\infty}^D$.
Moreover, we have
$$[E_{\infty}^D:E]=h_E \cdot t  \cdot \frac{1}{q-1} \cdot  \prod_{i=1}^s (q^{\deg(Q_i)}-1)q^{(c_i-1) \deg(Q_i)}.$$
First, we provide an upper bound for the order $m_F$ in terms of $g_F$ and the conductor.

\begin{lemma}
\label{orderofabg} Let $\mG$ be an abelian subgroup of $\Aut(F/\F_q)$ with $|\mG|=m_F$. Put $E=F^\mG$.
Let the divisor $D=\sum_{i=1}^{s}c_iQ_i$ be the conductor of $F/E$.
Put $t=\lceil  6 \log_q g_F+18\rceil$. Then
$$ \log_q m_F  \leq \log_q t+3g_E+\sum_{i=1}^s c_i \deg(Q_i).$$
\end{lemma}

\begin{proof}
First of all, by Hasse-Weil Theorem we have an upper bound on the class number $h_E$ (see \cite[Theorem 5.1.15 and Theorem 5.2.1]{St09})
 $$h_E\le (1+\sqrt{q})^{2g_E}.$$
Since $F$ is a subfield of $E_{\infty}^D$,  we have $[F:E]$ divides $[E_{\infty}^D:E]$, that is, 
$$m_F|h_E \cdot t \cdot \frac{1}{q-1} \cdot \prod_{i=1}^s (q^{\deg(Q_i)}-1)q^{(c_i-1) \deg(Q_i)}\le t(1+\sqrt{q})^{2g_E}q^{\sum_{i=1}^s c_i \deg(Q_i)}.$$
Hence, we obtain $$ \log_q m_F  \leq  \log_q t+2g_E \log_q(1+\sqrt{q})+ \sum_{i=1}^s c_i \deg(Q_i) \leq  \log_q t+3g_E+\sum_{i=1}^s c_i \deg(Q_i).$$
The proof is completed.\end{proof}
Lemma \ref{orderofabg} gives a relation between $m_F$ and conductor. This relation can be turned into a relation between   $m_F$ and $g_F$ as shown below.

\begin{proposition}
\label{gmrelation}
Let $F/\mathbb{F}_q$ be a global function field. Put $t=\lceil  6 \log_q g_F+18\rceil$. Then   we have
 $$g_F\geq \frac{1}{4} m_F\log_q m_F - \frac{1}{4}  m_F\log_q t -m_F+1. $$
\end{proposition}
\begin{proof} Let $\mG$ be an abelian subgroup of $\Aut(F/\F_q)$ with $|\mG|=m_F$. Let the divisor $D=\sum_{i=1}^{s}c_iQ_i$ be the conductor of $F/E$.
The Hurwitz genus formula \cite[Theorem 3.4.13]{St09} yields 
$$2g_F-2=m_F\cdot (2g_E-2) + \sum_{i=1}^s \sum_{P|Q_i} d(P|Q_i)  \deg(P).$$
As $d(P|Q_i)$, $f(P|Q_i)$ and $e(P|Q_i)$ depend only on $Q_i$, we may denote $d(P|Q_i)$, $f(P|Q_i)$ and $e(P|Q_i)$ by $d(Q_i)$, $f(Q_i)$ and $e(Q_i)$, respectively.
First of all, we have
\begin{eqnarray*} \sum_{P|Q_i} d(Q_i)  \deg(P)&=&\sum_{P|Q_i} d(Q_i) f(Q_i)  \deg(Q_i)\\
&=&\frac{d(Q_i)}{e(Q_i)}   \deg(Q_i) \sum_{P|Q_i} e(Q_i) f(Q_i)=\frac{d(Q_i)}{e(Q_i)}   \deg(Q_i)m_F.\end{eqnarray*}
The last equality holds from Fundamental Equality (see \cite[Theorem 3.1.11]{St09}).
Thus, 
$$ 2g_F-2=m_F\cdot (2g_E-2) +m_F\cdot \sum_{i=1}^s \frac{d(Q_i)}{e(Q_i)} \deg(Q_i).$$
Hence, we obtain
\begin{eqnarray*}
\frac{2g_F-2}{m_F} & =&2g_E-2+\sum_{i=1}^s \frac{d(Q_i)}{e(Q_i)} \deg(Q_i) \\
& \geq & 2g_E-2+\frac{1}{2} \sum_{i=1}^s c_i \deg(Q_i) \\
& = & \frac{3}{2} g_E+ \frac{1}{2} \sum_{i=1}^s c_i \deg(Q_i)+\frac{1}{2} g_E-2\\
& \geq & \frac{1}{2} \log_q m_F  - \frac{1}{2}  \log_q t -2.
\end{eqnarray*}
The first inequality follows from Equation \eqref{relationcde} in the subsection \ref{conductor} and
the last inequality follows from Lemma \ref{orderofabg}.
This completes the proof.
\end{proof}

Let us give a lower bound by considering cyclotomic function fields.

\begin{example}\label{ex:1}{\rm
Let $F$ be the cyclotomic function field $\mathbb{F}_q(x)(\Lambda_{P})$, where $P$ is an irreducible polynomial of degree $d$ in $\mathbb{F}_q[x]$. The Galois group of $F/\mathbb{F}_q(x)$ is a cyclic group of order $q^d-1$.
Then the maximum size $m_F$ of abelian subgroups of the automorphism group $\Aut(F/\F_q)$ is
$m_F\geq q^d-1$. By the Hurwitz genus formula and Proposition \ref{genusofcyclotomic}, the genus of $F$ is $$2g_F-2=(q^d-1)(-2)+(q^d-2)d+\frac{q-2}{q-1}(q^d-1).$$
Hence, this gives a lower bound on $M_q$:
$$M_q\geq \lim_{d\rightarrow \infty} \frac{m_F\cdot \log_q m_F}{g_F}= 2.$$
}\end{example}

The above example provides a  lower bound on $M_q$ for the case where $|\mG|$ is coprime to the characteristic. The next example gives the same lower bound on $M_q$ for the case where $|\mG|$ is divisible by the characteristic.

\begin{example}\label{ex:2}{\rm
Let $F_n$ be the cyclotomic function field $\mathbb{F}_q(x)(\Lambda_{P^n})$, where $P$ is an irreducible polynomial of degree $d$ in $\mathbb{F}_q[x]$ and $n\ge 2$. The Galois group of $F_n/\mathbb{F}_q(x)$ is an abelian group of order $(q^d-1)q^{d(n-1)}$.
Then the maximum size $m_{F_n}$ of abelian subgroups of the automorphism group  $\Aut(F_n/\F_q)$ is
$m_{F_n}\geq (q^d-1)q^{(n-1)d}$. 
By the Hurwitz genus formula and Proposition \ref{genusofcyclotomic}, the genus of $F_n$ is
$$2g_{F_n}-2=(q^d-1)q^{d(n-1)}(-2)+\big{[}n(q^d-1)q^{d(n-1)}-q^{d(n-1)}\big{]}d+\frac{q-2}{q-1}(q^d-1)q^{d(n-1)}.$$
Hence, this gives a lower bound on $M_q$:
$$M_q\geq \lim_{n\rightarrow \infty} \frac{m_{F_n} \cdot \log_q m_{F_n}}{g_{F_n}}= 2.$$
}\end{example}

\begin{theorem}
\label{upperbd}
For every prime power $q$, one has
$$2\leq M_q\leq 4.$$
\end{theorem}
\begin{proof} The lower bound $M_q\ge 2$ is given in Examples \ref{ex:1} and \ref{ex:2}. We now prove the upper bound.

Let $\{F/\F_q\}$ be a family of function fields with $g_F\rightarrow \infty$ and $M_q=\lim_{g_F\rightarrow \infty}\frac{m_F\log_q m_F}{g_F}$. By Examples \ref{ex:1} and \ref{ex:2}, one has that, for sufficiently large $g_F$,  $\frac{m_F\log_q m_F}{g_F}\ge M_q/2\ge 1$, i.e., $g_F\le m_F\log_qm_F$.

 Without loss of generality, we may assume that $g_F\ge 2$ for every function field $F$ in this family.
By Proposition \ref{gmrelation}, we have
 \begin{eqnarray*}g_F&\geq& \frac{1}{4} m_F\log_q m_F - \frac{1}{4}  m_F\log_q t -m_F+1\\
 &\geq& \frac{1}{4} m_F\log_q m_F - \frac{1}{4}  m_F \log_q(6\log_q g_F +18) -m_F+1\\
 &\ge & \frac{1}{4} m_F\log_q m_F - \frac{1}{4}  m_F \log_q(6\log_q (m_F\log_q m_F) +18) -m_F+1.
 \end{eqnarray*}
Dividing both sides of the above inequality by $m_F\log_qm_F$ and taking limits, one obtains $\lim_{g_F\rightarrow \infty}\frac{g_F}{m_F\log_q m_F}\ge \frac14$. The desired result follows.
\end{proof}

The upper bound given in the above theorem holds for both even and odd characteristics. However, for odd characteristic, we can refine the proof of Proposition  \ref{gmrelation} by better estimating the different exponent. Let us prove a lemma first.
\begin{lemma}
\label{clarger2}
Let $F/E$ be a finite abelian extension of global function fields and let $Q$ be a place of $E$ with conductor exponent $c=c_Q(F/E)\ge 2$. Let
$e=bp^w$ be the ramification index of $Q$ in the extension $F/E$, where $p$ is the characteristic of $F$, $p\nmid b$ and $w$ is a positive integer. Then the different exponent $d_Q(F/E)$ of $Q$ in $F/E$ has a lower bound
$$d_Q(F/E)\ge cbp^w-1-b-(c-2)bp^{w-1}.$$
\end{lemma}
\begin{proof}
Let $P$ be a place of $F$ lying over $Q$. 
Let $g_i$ be the order of $\mathcal G_i(P)$ and let $a$ be the least non-negative integer $k$ such that $\mathcal G_i(P)$ are trivial for all $i\ge k$. As the different exponent $d(P|Q)$ of $P|Q$ does not depend on the choice of $P$, we denote $d(P|Q)$ by $d_Q(F/E)$. 
Then the different exponent $d=d_Q(F/E)$ of $Q$ in $F/E$ is calculated by $$d=\sum_{i=0}^{a-1} (g_i-1)$$
from Hilbert's Different Theorem \cite[Theorem 3.8.7]{St09}.  The ramification theory of Galois extension yields $g_0=e=b p^w$ and $g_1=p^w$ \cite[Proposition 3.8.5]{St09}.
Let $n_j$ be the number of integers $i\ge 1$ with $g_i=p^{w-j+1}$ for $1\le j\le w$.
Furthermore, we have $$ a=1+\sum_{j=1}^w n_j, \quad d=ce-a=ce-1-\sum_{j=1}^w n_j.$$
The Hasse-Arf theorem \cite[Proposition 2.3.3]{NX01} shows that $g_0 | \sum_{i=1}^{n_1} g_i=n_1p^w,$ since $\G_{n_1}(P)\neq \G_{n_1+1}(P)$.
It follows that $b|n_1$ and $n_1\ge b$.
For $c\ge 2$, we have
\begin{eqnarray*}
ce&=&d+a=\sum_{i=0}^{a-1}g_i=g_0+\sum_{j=1}^w n_jp^{w-j+1}\\
&=&e+b\cdot p^w+(n_1-b)p^w+\sum_{j=2}^w n_jp^{w-j+1}\\
&\geq & 2e+(n_1-b)p+\sum_{j=2}^w n_jp=2e-bp+ \sum_{j=1}^w n_jp.
\end{eqnarray*}
It is clear that $b+(c-2)bp^{w-1}\geq \sum_{j=1}^w n_j$ from the above inequality and  $e=bp^w$.
Hence, the different exponent $d_Q(F/E)$ has a lower bound
$$d_Q(F/E)=d=ce-1-\sum_{j=1}^w n_j\ge cb p^w-1-b-(c-2)bp^{w-1}.$$ The proof is completed.
\end{proof}

\begin{theorem}
\label{oddasymptotic}
Assume that the characteristic $p$ of $\F_q$  is odd.
Let $F/\mathbb{F}_q$ be a global function field. Put $t=\lceil  6 \log_q g_F+18\rceil$. Then we have
 $$g_F\geq \frac{1}{3} m_F\log_q m_F -\frac{1}{3} m_F\log_q t-m_F+1. $$
As a consequence,  we have $M_q \leq 3.$
\end{theorem}

\begin{proof}
We use the same notations as in the proof of Proposition \ref{gmrelation}.
Let the divisor $D=\sum_{i=1}^r Q_i+\sum_{j=r+1}^s c_jQ_j$ with $c_j\ge 2$ be the conductor of $F/E$.
Then the Hurwitz genus formula yields
$$\frac{2g_F-2}{m_F}  =2g_E-2+\sum_{i=1}^s  \frac{d(Q_i)}{e(Q_i)} \deg(Q_i).$$
Suppose that there are exactly $r_0$ places of $E$ with ramification index $2$. Without loss of generality,  we can assume that
$$ \frac{d(Q_i)}{e(Q_i)}=1 - \frac{1}{e(Q_i)}=\frac{1}{2}$$ for $1\le i\le r_0$.
For other tamely ramified places $Q_i$ with $r_0+1\le i\le r$, we have
$$ \frac{d(Q_i)}{e(Q_i)}=1 - \frac{1}{e(Q_i)}\geq \frac{2}{3}.$$
For the wildly ramified places $Q_j$ with ramification index $e(Q_j)=b_jp^{w_j}$, the following inequality
$$\frac{d(Q_j)}{e(Q_j)}   \geq \frac{c_jb_jp^{w_j}-1-b_j-(c_j-2)p^{w_j-1}}{b_j p^{w_j}}  \geq \frac{2}{3} c_j$$
holds true for each $r+1\le j\le s$ from Lemma \ref{clarger2} and $p\ge 3$.

Since the extension $F/E$ is abelian, we have
$$m_F|h_E \cdot t \cdot \prod_{i=1}^s e(Q_i)  \leq q^{2g_E}\cdot t \cdot 2^{r_0} \cdot \prod_{i=r_0+1}^r(q^{\deg(Q_i)}-1)\cdot \prod_{j=r+1}^s (q^{\deg(Q_j)}-1)q^{(c_j-1)\deg(Q_j)}. $$
Hence, we have $$\log_q m_F - \log_q t \leq 2g_E+r_0\log_q2+\sum_{i=r_0+1}^r \deg(Q_i)+\sum_{j=r+1}^s c_j \deg(Q_j).$$
It follows that
\begin{eqnarray*}
\frac{2g_F-2}{m_F} & \geq & 2g_E-2+\frac{1}{2}\sum_{i=1}^{r_0} \deg(Q_i)+\frac{2}{3} \sum_{i=r_0+1}^r \deg(Q_i)+\frac{2}{3} \sum_{j=r+1}^s c_j \deg(Q_j) \\
& \geq & \frac{2}{3}\Big{[}2g_E+r_0\log_q 2+\sum_{i=r_0+1}^r \deg(Q_i)+\sum_{j=r+1}^s c_j \deg(Q_j)\Big{]} \\
& & + \frac{2g_E}{3}-2+\frac{1}{2} \sum_{i=1}^{r_0} \deg(Q_i)- \frac{2}{3}r_0 \log_q 2 \\
&  \geq  & \frac{2}{3}\log_q m_F-\frac{2}{3}\log_q t+\frac{2g_E}{3}-2+\frac{r_0}{6}\big{(}3-4\log_q 2\big{)}\\
&   \geq & \frac{2}{3}\log_q m_F-\frac{2}{3}\log_q t-2.
\end{eqnarray*}
Hence, we obtain  $$g_F\geq \frac{1}{3} m_F\log_q m_F -\frac{1}{3} m_F\log_q t-m_F+1. $$
The desired result on $M_q\le 3$ follows.
\end{proof}

\section{Asymptotic behavior of subgroups whose order is coprime to $q$}
The main purpose  of this section is to provide a linear lower bound by  constructing some tame towers of function fields which are recursively defined over $\F_q$.

\subsection{Towers of function fields}
First of all, let us introduce some basic definitions and results of towers of function fields. For more results on towers of function fields, please refer to \cite{BGS06, St09, St06}.
A tower of function fields over $\F_q$ is
an infinite sequence $\mathcal{F}=(F_0, F_1,F_2, \cdots)$ of function fields $F_n/\F_q$ with the following properties:
\begin{itemize}
\item[(i)] $F_0 \subsetneq F_1  \subsetneq  \cdots   \subsetneq F_n \subsetneq F_{n+1}\cdots $;
\item[(ii)] the extension $F_{n+1}/F_n$ is finite and separable for each $n\geq 0$;
\item[(iii)] The genera satisfy $g_{F_n}\rightarrow \infty$ for $n\rightarrow \infty$.
\end{itemize}
We say that the tower $\mathcal{F}=(F_0, F_1,F_2, \cdots)$ is tame if every place of $F_0$ is unramified or tamely ramified in the extension $F_n/F_0$ for any $n\ge 1$.

\begin{definition}
Let $\mathcal{F}=(F_0, F_1,F_2, \cdots)$ be an infinite sequence of function fields over $\F_q$ and let $f(T),h(T)\in \F_q(T)$ be two separable rational functions with $\deg(f)\ge 2$ (i.e., $\deg(f)=-v_{P_{\infty}}(f)$ where $v_{P_{\infty}}$ is the normalized discrete valuation at the infinite place).
We say that the sequence $\mathcal{F}$ can be described recursively by the equation $$f(Y)=h(X)$$
if there are elements $y_n$ for all $n\ge 0$ such that the followings hold true:
\begin{itemize}
\item [(1)] $F_0=\F_q(y_0)$ where $y_0$ is transcendental over $\F_q$;
\item [(2)] $F_{n+1}=F_n(y_{n+1})$ where $f(y_{n+1})=h(y_n)$ for every $n\ge 0$;
\item [(3)] $[F_{n+1}:F_n]=\deg (f)$ for every $n\ge 0$.
\end{itemize}
We define the corresponding basic function field $F$ of the sequence as
$$F:=\F_q(x,y) \text{ with } f(y)=h(x).$$
\end{definition}

\begin{definition}
Let $\mathcal{F}=(F_0, F_1,F_2, \cdots)$ be a tower over $\F_q$. 
\begin{itemize}
\item [(1)]  The  genus $\gamma(\mathcal{F}/F_0)$ of $\mathcal{F}$ over $F_0$ is defined by $$\gamma(\mathcal{F}/F_0)=\lim_{n\rightarrow \infty} \frac{g_{F_n}}{[F_n:F_0]}.$$
\item [(2)] The ramification locus of $\mathcal{F}$ over $F_0$ is defined by
$$\text{Ram}(\mathcal{F}/F_0)=\{ P\in \mathbb{P}_{F_0}| \ P \text{ is ramified in } F_n/F_0 \text{ for some } n\ge 1\}.$$
\end{itemize}
\end{definition}
Let $\mathcal{F}=(F_0, F_1,F_2, \cdots)$ be a tame tower of function fields over $\F_q$. Assume that $\text{Ram}(\mathcal{F}/F_0)$ is finite. For each place $P\in \text{Ram}(\mathcal{F}/F_0)$ and every place $Q\in \mathbb{P}_{F_n}$ lying over $P$, the different exponent $d(Q|P)=e(Q|P)-1$ since the tower $\mathcal{F}$ is tame.
Then the genus   $\gamma(\mathcal{F}/F_0)$ is finite and bounded by
$$ \gamma(\mathcal{F}/F_0)\le g_{F_0}-1+\frac{1}{2}  \sum_{P\in \text{Ram}(\mathcal{F}/F_0)} \deg P$$
from the Hurwitz genus formula (see \cite[Theorem 7.2.10]{St09}).

In the case of a finite ramification locus, the ramification divisor of $\mathcal{F}/F_0$ is defined by
$$R(\mathcal{F}/F_0)=\sum_{P\in \text{Ram}(\mathcal{F}/F_0)} P.$$
Let $L$ be an algebraic extension of $K$. Then we can consider the constant field extension $\mathcal{F}L=(F_0L, F_1L,F_2L, \cdots, F_iL,\cdots)$ of the tower $\mathcal{F}$ by $L$. The following proposition provides a useful criterion for the finiteness of the ramification locus of recursive towers \cite[Remark 7.2.22 and Proposition 7.2.23]{St09}.

\begin{proposition}
\label{ramificationlocus}
Let $\mathcal{F}=(F_0, F_1,F_2, \cdots)$ be a recursive tower over the finite field $\F_q$ defined by the equation $f(Y)=h(X)$.
We denote by $F$ the basic function field of $\mathcal{F}$. Let $L$ be a finite extension of $\F_q$ such that
all places of $L(x)$, which ramify in the extension $FL/L(x)$, are rational. Hence the set
$$\Lambda_0:=\{x(P)| P\in \mathbb{P}_{L(x)} \text{ is ramified in } FL/L(x)\}$$
is contained in $L\cup \{\infty\}$. Suppose that there exists a finite subset $\Lambda \supseteq \Lambda_0$ of $L\cup \{\infty\}$ such that
any solution $\a \in \overline{\F_q}\cup \{\infty\}$ of the equation $h(\a)=f(\b)$ for every $\b\in \Lambda$ is still in $\Lambda$.
Then the ramification locus $\text{Ram}(\mathcal{F}L/F_0L)$ is finite and $$\text{Ram}(\mathcal{F}L/F_0L)\subseteq \{P\in \mathbb{P}_{F_0L} | \ y_0(P)\in \Lambda\}.$$
Furthermore, $\text{Ram}(\mathcal{F}/F_0)$ is finite and  $\deg  R(\mathcal{F}/F_0)=\deg  R(\mathcal{F}L/F_0L)$.
\end{proposition}

\subsection{Constructions of tame towers}
In this subsection, we will provide a lower bound ${B_q}\ge 2/3$ by explicitly constructing two tame towers of function fields.

\begin{proposition}
\label{examplenot3}
Assume that $\text{char}(\F_q)\neq 3$. The infinite sequence $\mathcal{F}=(F_0, F_1,F_2, \cdots)$ which is recursively defined over $\F_q$ by the equation $$Y^3=\frac{X^2+X+1}{3X}$$ is a tame tower of function fields over $\F_q$.
\end{proposition}
\begin{proof}
First note that  $F_{n+1}=F_n(y_{n+1})$ with $y_{n+1}^3=(y_n^2+y_n+1)/(3y_n)$, hence $[F_{n+1}:F_n]\leq 3$.
Since $\text{char}(\F_q)\neq 3$, it follows that $F_{n+1}/F_n$ is separable.

The goal is to find places $P_n \in  \mathbb{P}_{F_n}$ and  $P_{n+1} \in  \mathbb{P}_{F_{n+1}}$ with $P_{n+1}|P_n$ such that $e(P_{n+1}|P_n)=3.$ We proceed as follows.
Let $P_0\in \mathbb{P}_{F_0}$ be the unique pole of $y_0$ in $F_0=\F_q(y_0)$ and let $P_1\in \mathbb{P}_{F_1}$ lie over $P_0$. From the equation $y_{1}^3=(y_0^2+y_0+1)/(3y_0)$, we obtain
$$3v_{P_1}(y_1)=v_{P_1}(y_1^3)=e(P_1|P_0)v_{P_0}\Big{(}\frac{y_0^2+y_0+1}{3y_0}\Big{)}=e(P_1|P_0) \cdot (-1).  $$
Hence, $e(P_1|P_0)=3$ and $v_{P_1}(y_1)=-1.$ Similarly from the equation $y_{2}^3=(y_1^2+y_1+1)/(3y_1)$, there exists a place $P_2\in \mathbb{P}_{F_2}$ with $P_2|P_1$ such that $e(P_2|P_1)=3$ and $v_{P_2}(y_2)=-1.$ By iterating this process, we obtain $P_{n+1}\in\mathbb{P}_{F_{n+1}}$ with $P_{n+1}|P_n$ such that $e(P_{n+1}|P_n)=3$ and $v_{P_{n+1}}(y_{n+1})=-1$ for all $n\geq 0$. Therefore,
$[F_{n+1}:F_n]=e(P_{n+1}|P_n)=3$.  It follows that $F_{n+1}$ and $F_n$ have the same full constant field $\F_q$, since constant field extensions are unramified (see \cite[Proposition 7.2.15]{St09}). Hence, the conditions (i) and (ii) for a tower have been showed.

From the theory of Kummer extension \cite[Proposition 3.7.3]{St09}, exactly the following places of $F_0$ are ramified in $F_1/F_0$: the zero and the pole of $y_0$, the two zeros of $y_0^2+y_0+1$ or a place of degree $2$ depending on the factorization of $y_0^2+y_0+1$ in $\F_q[y_0]$.
The Hurwitz genus formula for $F_1/F_0$ yields $2g_{F_1}-2=3\times (2g_{F_0}-2)+4\times (3-1).$ Hence, $g_{F_1}=2$.
 
 As $F_i\subsetneq F_{i+1}$, we have \[2g_{F_{i+1}}-2\ge [F_{i+1}:F_i](2g_{F_i}-2)\ge 2(2g_{F_i}-2)= 4g_{F_i}-4.\]
Thus, if $g_{F_i}\ge 2$, we have
$2g_{F_{i+1}}\ge 4g_{F_i}-2\ge 3g_{F_i}$, i.e., $g_{F_{i+1}}\ge 3g_{F_i}/2>g_{F_i}$. This implies that 
 the genus $g_{F_i}$ strictly increases for $i\ge 1$.  The condition (iii) for a  tower of function fields  is satisfied.
\end{proof}

\begin{theorem}
\label{toweroverff}
Assume that $\text{char}(\F_q)\neq 3$.
Let $\mathcal{F}=(F_0,F_1,\cdots)$ be a recursive tower defined by the equation $$Y^3=\frac{X^2+X+1}{3X}$$ over the finite field $\F_q$. Let $E_n$ be the Galois closure of $F_n$ over $F_0$. Then
$${B_q} \ge \limsup_{n\rightarrow \infty} \frac{|\Gal(E_n/F_0)|}{g_{E_n}}\geq \frac{2}{3}.$$
\end{theorem}
\begin{proof}
First let us choose a finite field $L\supseteq \F_q$ such that $L$ contains an element $w$ of order $3$.
Note that if $3|(q-1)$, then we can choose $L=\F_q$, otherwise we let $L=\F_{q^2}$.
From the theory of Kummer extension, all ramified places in $FL/L(x)$ are exactly the rational places $P_0,P_{w}, P_{w^2}$ and $P_{\infty}$, that is, $\Lambda_0=\{0, w, w^2, \infty\}.$
Now let us consider the set $$\Lambda=\{0, w, w^2,\infty, 1\} \subseteq L\cup \{\infty\}$$ which satisfies the condition of Proposition \ref{ramificationlocus}. In fact, it remains to show that any solution $\a \in \overline{\F_q}\cup \{\infty\}$ of the equation $$\frac{\a^2+\a+1}{3\a}=\b^3$$ for every $\b\in \Lambda$ is still in $\Lambda$. This can be verified  easily as follows:
\begin{itemize}
\item if $\b=\infty$, then $\a=\infty$ or $\a=0$;
\item if $\b=0$, then $\a=w \text{ or } \a=w^2$;
\item if $\b=1, w \text{ or } w^2$, then $\a=1.$
\end{itemize}
Hence, it has been shown that the ramification locus
$$\text{Ram}(\mathcal{F}L/F_0L)\subseteq \{P\in \mathbb{P}_{F_0L} | y_0(P)\in \Lambda\}=\{P_0, P_w, P_{w^2}, P_1, P_{\infty}\}$$
from Proposition \ref{ramificationlocus}.

Let $E_n$ be the Galois closure of $F_n$ over $F_0$. 
In fact, $E_n$ is the compositum of the fields $\sigma(F_n)$, where $\sigma$ runs through all embeddings $\sigma:F_n \rightarrow \overline{F_0}$ (note that $\overline{F_0}\supseteq F_0$ is the algebraic closure of $F_0$). 
If $P$ is unramified in $F_n/F_0$, then $P$ is unramified in $E_n/F_0$ from an immediate consequence of Abhyankar's Lemma (see \cite[Corollary 3.9.3]{St09}). 
As for $P$ is tamely ramified in $F_n/F_0$, it is also tamely ramified in $\sigma(F_n)/F_0$. Hence,
$P$ is tamely ramified in $E_n/F_0$ from Abhyankar's Lemma \cite[Theorem 3.9.1]{St09}. 
Let $\mathcal{E}=(F_0, E_1, \cdots, E_n, \cdots)$ be the Galois closure of the tower $\mathcal{F}$. Then the tower $\mathcal{E}$ is tame and $\text{Ram}(\mathcal{E}/F_0)=\text{Ram}(\mathcal{F}/F_0)$. 

The genus $ \gamma(\mathcal{E}/F_0)$ of $\mathcal{E}$  over $F_0$ is
\begin{eqnarray*}
\gamma(\mathcal{E}/F_0) & \leq  & g_{F_0}-1+\frac{1}{2}\sum_{P\in \text{Ram}(\mathcal{E}/F_0)} \deg (P) \\ 
&=& g_{F_0}-1+\frac{1}{2}\sum_{P\in \text{Ram}(\mathcal{F}/F_0)} \deg (P) \\ 
& = & g_{F_0}-1+\frac{1}{2}\sum_{P\in \text{Ram}(\mathcal{F}L/F_0L)} \deg (P)\\ &\le & \frac{3}{2}.
\end{eqnarray*}
The first inequality follows from \cite[Proposition 7.2.10]{St09} and the second equality follows from Proposition \ref{ramificationlocus}.
Let $b_{E_n}$ be the maximum size of subgroups of $\Aut(E_n/\F_q)$ whose order is coprime to $q$.
The Galois group of $E_n/F_0$ is a subgroup of  $\Aut(E_n/\F_q)$.
It follows that
$$b_{E_n} \geq |\Gal(E_n/F_0)|=[E_n:F_0].$$
Hence, we obtain $${B_q}\geq \limsup_{n\rightarrow \infty} \frac{b_{E_n}}{g_{E_n}}\ge \limsup_{n\rightarrow \infty} \frac{[E_n:F_0]}{g_{E_n}}= \frac{1}{\gamma(\mathcal{E}/F_0)} \ge \frac{2}{3}.$$
\end{proof}

\begin{theorem}
Assume that $\text{char}(\F_q)\neq 2$.
Let $\mathcal{F}=(F_0,F_1,\cdots)$ be a recursive tower defined by the equation $$Y^4=\frac{X^2+1}{2X}$$ over the finite field $\F_q$. Let $E_n$ be the Galois closure of $F_n$ over $F_0$. Then we have
$${B_q} \ge \limsup_{n\rightarrow \infty} \frac{|\Gal(E_n/F_0)|}{g_{E_n}}\geq \frac{2}{3}.$$
\end{theorem}
\begin{proof}
It is easy to verify that $\Lambda_0=\{0,i, -i, \infty\}$ and $\Lambda=\{0,i, -i, \infty, 1\}$  satisfy the conditions of Proposition \ref{ramificationlocus}.
The rest of the proof uses the similar arguments as in  Proposition \ref{examplenot3} and Theorem \ref{toweroverff}. The details are omitted.
\end{proof}

\end{document}